\documentclass[11pt]{amsart}
\usepackage{verbatim}
\usepackage{amsmath}
\usepackage{enumerate}
\usepackage{amssymb}
\usepackage{color}
\usepackage{url}
\usepackage{graphicx}
\usepackage{fullpage}

\newcommand\+{\;\lower\plusheight\hbox{$+$}\;}
\newcommand\lldots{\;\lower\plusheight\hbox{$\cdots$}\;}

\newtheorem{Theorem}{Theorem}[section]
\newtheorem{Lemma}[Theorem]{Lemma}

\setbox0=\hbox{$+$}
\newdimen\plusheight
\plusheight=\ht 0 \setbox0=\hbox{$-$}
\newdimen\minusheight
\minusheight=\ht0

\setbox0=\hbox{$\cdots$}
\newdimen\cdotsheight
\cdotsheight=\plusheight

\begin{document}

\title[Representations by $2k$-ary quadratic forms]{Representations of integers by certain $2k$-ary quadratic forms}
\author{Dongxi Ye}
\address{
Department of Mathematics, University of Wisconsin\\
480 Lincoln Drive, Madison, Wisconsin, 53706 USA\\
E-mail: lawrencefrommath@gmail.com}
\maketitle
\begin{abstract}
Suppose $k$ is a positive integer. In this work, we establish formulas for for the number of representations of integers by the quadratic forms
$$
x_{1}^{2}+\cdots+x_{k}^{2}+m\left(x_{k+1}^{2}+\cdots+x_{2k}^{2}\right)
$$
for $m\in\{2,4\}$. 
\end{abstract}
\noindent
\numberwithin{equation}{section}
\allowdisplaybreaks
\section{Introduction}
\label{intro}

In the long history of number theory, one of the classical problems is to give an explicit formula for the number of ways that one can express a positive integer $n$ as a sum of $2k$ squares, that is, the number of integral solutions of
$$
x_{1}^{2}+x_{2}^{2}+\cdots+x_{2k-1}^{2}+x_{2k}^{2}=n,
$$
which we denote by $\mathcal{R}_{k}(n)$. It is known from the theory of modular forms that in general, 
$$
\mathcal{R}_{k}(n)=\Delta_{k}(n)+\mathcal{E}_{k}(n)
$$
where $\Delta_{k}(n)$ is a divisor function and $\mathcal{E}_{k}(n)$ is a function of order substantially lower than that of $\Delta_{k}(n)$. Formulas for $\mathcal{R}_{k}(n)$ in this fashion have been found and studied by various mathematicians. For $k=1,\,2,\,3$ and $4$, i.e., sums of $2,\,4,\,6$ and $8$ squares, (reformulated) formulas for $\mathcal{R}_{k}(n)$ are originally due to Jacobi \cite{jaco},
\begin{align}
\label{R1}
\mathcal{R}_{1}(n)&=4\sum_{d|n}\left(\frac{-4}{d}\right),\\
\mathcal{R}_{2}(n)&=8\sum_{d|n}d-32\sum_{d|\frac{n}{4}}d,\\
\mathcal{R}_{3}(n)&=-4\sum_{d|n}\left(\frac{-4}{d}\right)d^{2}+16\sum_{d|n}\left(\frac{-4}{n/d}\right)d^{2},\\
\label{R4}
\mathcal{R}_{4}(n)&=16\sum_{d|n}d^{3}-32\sum_{d|\frac{n}{2}}d^{3}+256\sum_{d|\frac{n}{4}}d^{3}
\end{align}
where, here and throughout this work, $\left(\frac{\cdot}{\cdot}\right)$ denotes the Kronecker symbol. The result for $k=5$, i.e., sum of $10$ squares, was due (without proof) in part to Eisenstein \cite{eis}, and fully described (without proof) by Liouville \cite{liou}. The results for $1\leq k\leq9$ were all proved by Glaisher \cite{gla}. In around 1916, this classical problem was ``completely'' solved (without proof) by Ramanujan \cite{ram}, \cite[Eqs. (145)--(147)]{ram2}. To state Ramanujan's result, we need the well-known Dedekind eta function
$$
\eta(\tau):=q^{1/24}\prod_{j=1}^{\infty}(1-q^{j})
$$
where, here and throughout this paper, $\tau$ denotes a complex number with positive imaginary part and $q=e^{2\pi i\tau}$. For brevity, throughout this work, we write $\eta_{m}$ for $\eta(m\tau)$ for any positive integer $m$. In addition, we define $\chi_{D}(\cdot)$ to be the Kronecker symbol $\left(\frac{D}{\cdot}\right)$, and define $\sigma_{k}(n)$, $\sigma_{k,\chi_{D}}^{\infty}(n)$ and $\sigma_{k,\chi_{D}}^{0}(n)$ to be the divisor functions
\begin{align*}
\sigma_{k}(n)&=\sum_{d|n}d^{k},\\
\sigma_{k,\chi_{D}}^{\infty}(n)&=\sum_{d|n}\chi_{D}(d)d^{k},\\
\sigma_{k,\chi_{D}}^{0}(n)&=\sum_{d|n}\chi_{D}(n/d)d^{k}
\end{align*}
with the convention that they are defined to be 0 if $n$ is not a positive integer. Now we reformulate and summarize Ramanujan's result in Theorem \ref{rammor} below.

\begin{Theorem}[Ramanujan]
\label{rammor}
Suppose $k$ is a positive integer. Then there are unique rational numbers $c_{j,k}$ depending on $j$ and $k$ such that
\begin{align*}
&\mathcal{R}_{k}(n)\\
&=\begin{cases}{\displaystyle-\frac{2k}{B_{k}}\left(\frac{(-1)^{k/2}\sigma_{k-1}(n)-(1+(-1)^{k/2})\sigma_{k-1}(n/2)+2^{k}\sigma_{k-1}(n/4)}{-1+2^{k}}\right)}&\mbox{if $k$ is even,}\\{\displaystyle-\frac{2k}{B_{k,4}}\left(\frac{\sigma_{k-1,\chi_{-4}}^{\infty}(n)+(-1)^{(k-1)/2}2^{k-1}\sigma_{k-1,\chi_{-4}}^{0}(n/4)}{1+\delta_{k,1}}\right)}&\mbox{if $k$ is odd}
\end{cases}\\
&\quad+\sum_{1\leq j\leq(k-1)/4}c_{j,k}a_{j,k}(n)
\end{align*}
where, here and throughout this paper, $\delta_{\cdot,\cdot}$ denotes the Kronecker delta, $B_{k}$ is the $k$th ordinary Bernoulli number, $B_{k,4}$ is the $k$th generalized Bernoulli number of order $4$ defined via
\begin{equation*}
\frac{t}{e^{4t}-1}\sum_{j=1}^{4}\chi_{-4}(j)e^{jt}=\sum_{n=0}^{\infty}B_{n,4}\frac{t^{n}}{n!},
\end{equation*}
and the numbers $a_{j,k}(n)$ are defined via
$$
\sum_{n=0}^{\infty}a_{j,k}(n)q^{n}=\frac{\eta_{2}^{10k}}{(\eta_{1}\eta_{4})^{4k}}\times\frac{(\eta_{1}\eta_{4})^{24j}}{\eta_{2}^{48j}}.
$$
\end{Theorem}
Theorem \ref{rammor} was proved first by Mordell \cite{mordell} utilizing the theory of modular forms. An elementary proof was given by Cooper in \cite{cooper} by making skillful use of Ramanujan's ${}_{1}\psi_{1}$ formula. 

In recent work \cite{cky}, Cooper, Kane and the author extended Ramanujan's results and determined equivalently the number of integral solutions of 
$$
x_{1}^{2}+\cdots+x_{k}^{2}+p\left(x_{k+1}^{2}+\cdots+x_{2k}^{2}\right)=n
$$
for $p\in\{3,7,11,23\}$ along the ``divisor function + lower order term'' fashion. Motivated by the work of Ramanujan, and Cooper \emph{et al.}, in this work we aim to establish analogous formulas for the number of representations of integers by the quadratic forms
$$
x_{1}^{2}+\cdots+x_{k}^{2}+m\left(x_{k+1}^{2}+\cdots+x_{2k}^{2}\right)
$$
for $m\in\{2,4\}$. 

This work is organized as follows. In Section \ref{defm}, we state our main results, and as illustrations, we also present some examples that follow from the general case we obtain. Proofs will be given in Section~\ref{pfs}. In the last section, we conclude this work with some remarks, which explain the existence of these Ramanujan-Mordell type formulas.

\section{Statement of Results}
\label{defm}

Let us denote by $r(1^{k}m^{k};n)$ the number of integral solutions of the equation
$$
x_{1}^{2}+\cdots+x_{k}^{2}+m\left(x_{k+1}^{2}+\cdots+x_{2k}^{2}\right)=n.
$$
The main results of this work are summarized in the following theorem.

\begin{Theorem}
\label{main}
Suppose $k$ is a positive integer.
\begin{enumerate}
\item{Let $\ell_{2}$ be defined by
$$
\ell_{2}=\begin{cases}\frac{k-1}{2}&\mbox{if $k$ is odd,}\\
\frac{k-2}{2}&\mbox{if $k$ is even.}
\end{cases}
$$
Then there are unique rational numbers $c_{j,k,2}$ depending on $j$ and $k$ such that
\begin{align*}
&r(1^{k}2^{k};n)\\
&=\begin{cases}
{\displaystyle-\frac{k}{B_{k,8}}\left(\frac{2\sigma_{k-1,\chi_{-2}}^{\infty}(n)+2(-8)^{(k-1)/2}\sigma_{k-1,\chi_{-2}}^{0}(n)}{1+\delta_{k,1}}\right)}&\mbox{if $k$ is odd,}\\
{\displaystyle-\frac{k}{B_{k}}\left(\frac{(-1)^{\frac{k}{2}}\sigma_{k-1}(n)-(-1)^{\frac{k}{2}}\sigma_{k-1}\left(\frac{n}{2}\right)-2^{\frac{k}{2}}\sigma_{k-1}\left(\frac{n}{4}\right)+8^{\frac{k}{2}}\sigma_{k-1}\left(\frac{n}{8}\right)}{2^{\frac{k}{2}-1}(2^{k}-1)}\right)}&\mbox{if $k$ is even}
\end{cases}\nonumber\\
&\quad+\sum_{j=1}^{\ell_{2}}c_{j,k,2}a_{j,k,2}(n)\nonumber
\end{align*}
where $\chi_{-2}$, $\sigma_{k}(n)$, $\sigma_{k,\chi_{-2}}^{\infty}(n)$ and $\sigma_{k,\chi_{-2}}^{0}(n)$ are as defined in Section \ref{intro}, and $B_{k,8}$ is the $k$th generalized Bernoulli of order $8$ defined via
\begin{equation*}
\frac{t}{e^{8t}-1}\sum_{j=1}^{8}\chi_{-2}(j)e^{jt}=\sum_{n=0}^{\infty}B_{n,8}\frac{t^{n}}{n!},
\end{equation*}
and the numbers $a_{j,k,2}(n)$ are given by
$$
\sum_{n=0}^{\infty}a_{j,k,2}(n)q^{n}=\frac{(\eta_{2}\eta_{4})^{3k}}{(\eta_{1}\eta_{8})^{2k}}\times\left(\frac{\eta_{1}\eta_{8}}{\eta_{2}\eta_{4}}\right)^{8j}.
$$
}

\item{Let $\ell_{4}$ be defined by
$$
\ell_{4}=\begin{cases}k-2&\mbox{if $k$ is odd,}\\
 k-1&\mbox{if $k$ is even.}
\end{cases}
$$
Then there are unique rational numbers $c_{j,k,4}$ depending on $j$ and $k$ such that
\begin{align*}
&r(1^{k}4^{k};n)\\
&=\begin{cases}
\quad{2\sigma_{0}^{\infty}(n)-2\sigma_{0}^{\infty}\left(\frac{n}{2}\right)+4\sigma_{0}^{\infty}\left(\frac{n}{4}\right)}&\mbox{if $k=1$,}\\
{\displaystyle-\frac{k}{B_{k,4}}\left({{(-1)^{\frac{k+1}{2}}\sigma_{k-1}^{\infty}(n)-(-1)^{\frac{k+1}{2}}\sigma_{k-1}^{\infty}\left(\frac{n}{2}\right)+2\sigma_{k-1}^{\infty}\left(\frac{n}{4}\right)\quad\atop\quad-(-1)^{\frac{k+1}{2}}\sigma_{k-1}^{0}(n)+2^{k-1}\sigma_{k-1}^{0}\left(\frac{n}{2}\right)-2^{2k-1}\sigma_{k-1}^{0}\left(\frac{n}{4}\right)}}\right)}&\mbox{if $k\geq3$ and $k$ is odd,}\\
{\displaystyle-\frac{k}{B_{k}}\left(\frac{(-1)^{\frac{k}{2}}\sigma_{k-1}(n)-(-1)^{\frac{k}{2}}\sigma_{k-1}\left(\frac{n}{2}\right)-2^{k}\sigma_{k-1}\left(\frac{n}{8}\right)+4^{k}\sigma_{k-1}\left(\frac{n}{16}\right)}{2^{k-1}(2^{k}-1)}\right)}&\mbox{if $k$ is even}
\end{cases}\nonumber\\
&\quad+\sum_{j=1}^{\ell_{4}}c_{j,k,4}a_{j,k,4}(n)
\end{align*}
where the divisor functions $\sigma_{k}(n)$, $\sigma_{k,\chi_{-4}}^{\infty}(n)$ and $\sigma_{k,\chi_{-4}}^{0}(n)$ are as defined in Section \ref{intro}, and the numbers $a_{j,k,4}(n)$ are given by
$$
\sum_{n=0}^{\infty}a_{j,k,4}(n)q^{n}=\frac{(\eta_{2}\eta_{8})^{5k}}{(\eta_{1}\eta_{4}^{2}\eta_{16})^{2k}}\times\frac{(\eta_{1}\eta_{4}\eta_{16})^{4j}}{(\eta_{2}\eta_{8})^{6j}}.
$$
}
\end{enumerate}
\end{Theorem}

For $k=1$, 2, 3 or 4, Theorem \ref{main} gives the following analogues of \eqref{R1}--\eqref{R4}.
\begin{align}
\label{12}
r(1^{1}2^{1};n)&=2\sum_{d|n}\left(\frac{-2}{d}\right),\\
\label{14}
r(1^{1}4^{1};n)&=2\sum_{d|n}\left(\frac{-4}{d}\right)-2\sum_{d|\frac{n}{2}}\left(\frac{-4}{d}\right)+4\sum_{d|\frac{n}{4}}\left(\frac{-4}{d}\right),\\
\label{1122}
r(1^{2}2^{2};n)&=4\sum_{d|n}d-4\sum_{d|\frac{n}{2}}d+8\sum_{d|\frac{n}{4}}d-32\sum_{d|\frac{n}{8}}d,\\
\label{1144}
r(1^{2}4^{2};n)&=2\sum_{d|n}d-2\sum_{d|\frac{n}{2}}d+8\sum_{d|\frac{n}{8}}d-32\sum_{d|\frac{n}{16}}d+2a_{1,2,4}(n),\\
\label{1323}
r(1^{3}2^{3};n)&=-\frac{2}{3}\sum_{d|n}\left(\frac{-2}{d}\right)d^{2}+\frac{16}{3}\sum_{d|n}\left(\frac{-2}{n/d}\right)d^{2}+\frac{4}{3}a_{1,3,2}(n),\\
\label{1343}
r(1^{3}4^{3};n)&=-2\sum_{d|n}\left(\frac{-4}{d}\right)d^{2}+2\sum_{d|\frac{n}{2}}\left(\frac{-4}{d}\right)d^{2}-4\sum_{d|\frac{n}{4}}\left(\frac{-4}{d}\right)d^{2}\\
&\quad+2\sum_{d|n}\left(\frac{-4}{n/d}\right)d^{2}-8\sum_{d|\frac{n}{2}}\left(\frac{-4}{n/2d}\right)d^{2}+64\sum_{d|\frac{n}{4}}\left(\frac{-4}{n/4d}\right)d^{2}+6a_{1,3,4}(n),\nonumber\\
\label{1424}
r(1^{4}2^{4};n)&=4\sum_{d|n}d^{3}-4\sum_{d|\frac{n}{2}}d^{3}-16\sum_{d|\frac{n}{4}}d^{3}+64\sum_{d|\frac{n}{8}}d^{3}+4a_{1,4,2}(n),\\
\label{1444}
r(1^{4}4^{4};n)&=\sum_{d|n}d^{3}-\sum_{d|\frac{n}{2}}d^{3}-16\sum_{d|\frac{n}{8}}d^{3}+256\sum_{d|\frac{n}{16}}d^{3}+7a_{1,4,4}(n)-12a_{2,4,4}(n).
\end{align}
The formula \eqref{12} was proved by Shen \cite{shen}. The formula \eqref{14} was due in part to Ramanujan \cite[Entry 25(i), (iii), p. 40]{bru}, \cite[Entry 18, p. 152]{bru2}. The formulas \eqref{1122} and \eqref{1144} were stated without proof by Liouville \cite{liou3, liou2} and proved by Pepin \cite{pe1,pe2}, Bachmann \cite{ba}, and Alaca \emph{et al.} \cite{ala} . The formulas \eqref{1323}, and \eqref{1343}--\eqref{1444} are due to Alaca \emph{et al.} \cite{ala2} and Alaca \emph{et al.} \cite{ala3,ala4,ala5}, respectively.

Now if we consider the generating function of $r(1^{k}m^{k};n)$, we note that
$$
\sum_{n=0}^{\infty}r(1^{k}m^{k};n)q^{n}=\left(\theta(\tau)\theta(m\tau)\right)^{k}
$$
where $\theta(\tau)$ is Ramanujan's theta function defined by
\begin{equation}
\label{thetaeta}
\theta(\tau)=\sum_{n=-\infty}^{\infty}q^{n^{2}}=\frac{\eta_{2}^{5}}{\eta_{1}^{2}\eta_{4}^{2}}
\end{equation}
where the $\eta$-quotient representation after the second equality is due to Jacobi \cite{jaco}. In view of that $r(1^{k}m^{k};n)$ is the $n$th Fourier coefficient of $\left(\theta(\tau)\theta(m\tau)\right)^{k}$ and the $\eta$-quotient representation \eqref{thetaeta}, Theorem \ref{main} is equivalent to the following theorem.

\begin{Theorem}
\label{main1}
Suppose $k$ is a positive integer. Let $E_{k}(\tau)$ be the normalized Eisenstein series of weight $k$ on $\mbox{SL}_{2}(\mathbb{Z})$ defined by
\begin{equation}
E_{k}(\tau)=1-\frac{2k}{B_{k}}\sum_{n=1}^{\infty}\sigma_{k-1}(n)q^{n}.\nonumber
\end{equation}

\begin{enumerate}
\item{
Let $\ell_{2}$ be as defined in Theorem \ref{main}(1). Let $F_{k,2}(\tau)$ be defined by
$$
F_{k,2}(\tau)=\begin{cases}
{\displaystyle\frac{E_{k,\chi_{-2}}^{\infty}(\tau)+(-8)^{(k-1)/2}E_{k,\chi_{-2}}^{0}(\tau)}{1+\delta_{k,1}}}&\mbox{if $k$ is odd,}\\
{\displaystyle\frac{(-1)^{\frac{k}{2}}E_{k}(\tau)-(-1)^{\frac{k}{2}}E_{k}(2\tau)-2^{\frac{k}{2}}E_{k}(4\tau)+8^{\frac{k}{2}}E_{k}(8\tau)}{2^{\frac{k}{2}}(2^{k}-1)}}&\mbox{if $k$ is even,}
\end{cases}
$$
where $E_{k,\chi_{-2}}^{\infty}(\tau)$ and $E_{k,\chi_{-2}}^{0}(\tau)$ are Eisenstein series of weight $k$ on $\Gamma_{0}(8)$ with character $\chi_{-2}$ defined by
\begin{align*}
E_{k,\chi_{-2}}^{\infty}(\tau)&=1-\frac{2k}{B_{k,8}}\sum_{n=1}^{\infty}\sigma_{k-1,\chi_{-2}}^{\infty}(n)q^{n}
\intertext{and}
E_{k,\chi_{-2}}^{0}(\tau)&=\delta_{k,1}-\frac{2k}{B_{k,8}}\sum_{n=1}^{\infty}\sigma_{k-1,\chi_{-2}}^{0}(n)q^{n},
\end{align*}
and let $x_{2}=x_{2}(\tau)$ be defined by
$$
x_{2}(\tau)=\left(\frac{\eta_{1}\eta_{8}}{\eta_{2}\eta_{4}}\right)^{8}.
$$
Then  there are unique rational numbers $c_{j,k,2}$ depending on $j$ and $k$ such that
\begin{equation}
\label{eq1}
\left(\theta(\tau)\theta(2\tau)\right)^{k}=F_{k,2}(\tau)+{\left(\theta(\tau)\theta(2\tau)\right)^{k}\sum_{j=1}^{\ell_{2}}c_{j,k,2}x_{2}^{j}}.
\end{equation}
}

\item{
Let $\ell_{4}$ be as defined in Theorem \ref{main}(2). Let $F_{k,4}(\tau)$ be defined by
$$
F_{k,4}(\tau)=\begin{cases}
{\displaystyle\frac{1}{2}{\left(E_{1,\chi_{2}}^{\infty}(\tau)-E_{1,\chi_{2}}^{\infty}(2\tau)+2E_{1,\chi_{-2}}^{\infty}(4\tau)\right)}}&\mbox{if $k=1$,}\\
{\displaystyle\frac{1}{2}\left({{(-1)^{\frac{k+1}{2}}E_{k,\chi_{-2}}^{\infty}(\tau)-(-1)^{\frac{k+1}{2}}E_{k,\chi_{-2}}^{\infty}(2\tau)+2E_{k,\chi_{-2}}^{\infty}(4\tau)\qquad\atop-(-1)^{\frac{k+1}{2}}E_{k,\chi_{-2}}^{0}(\tau)+2^{k-1}E_{k,\chi_{-2}}^{0}(2\tau)-2^{2k-1}E_{k,\chi_{-2}}^{0}(4\tau)}}\right)}&\mbox{if $k\geq3$ and is odd,}\\
{\displaystyle\frac{(-1)^{\frac{k}{2}}E_{k}(\tau)-(-1)^{\frac{k}{2}}E_{k}(2\tau)-2^{k}E_{k}(8\tau)+4^{k}E_{k}(16\tau)}{2^{k}(2^{k}-1)}}&\mbox{if $k$ is even,}
\end{cases}
$$
where $E_{k,\chi_{-4}}^{\infty}(\tau)$ and $E_{k,\chi_{-4}}^{0}(\tau)$ are Eisenstein series of weight $k$ on $\Gamma_{0}(4)$ with character $\chi_{-4}$ defined by
\begin{align*}
E_{k,\chi_{-4}}^{\infty}(\tau)&=1-\frac{2k}{B_{k,4}}\sum_{n=1}^{\infty}\sigma_{k-1,\chi_{-4}}^{\infty}(n)q^{n}
\intertext{and}
E_{k,\chi_{-4}}^{0}(\tau)&=\delta_{k,1}-\frac{2k}{B_{k,4}}\sum_{n=1}^{\infty}\sigma_{k-1,\chi_{-4}}^{0}(n)q^{n},
\end{align*}
and let $x_{4}=x_{4}(\tau)$ be defined by
$$
x_{4}(\tau)=\frac{(\eta_{1}\eta_{4}\eta_{16})^{4}}{(\eta_{2}\eta_{8})^{6}}.
$$
Then there are unique rational numbers $c_{j,k,4}$ depending on $j$ and $k$ such that
\begin{equation}
\label{eq2}
\left(\theta(\tau)\theta(4\tau)\right)^{k}=F_{k,4}(\tau)+\left(\theta(\tau)\theta(4\tau)\right)^{k}\sum_{j=1}^{\ell_{4}}c_{j,k,4}x_{4}^{j}.
\end{equation}
}
\end{enumerate}
\end{Theorem}

\section{Proof of Results}
\label{pfs}
This section is devoted to proving Theorem \ref{main1}. The proof hinges on the following preliminary results.

\begin{Lemma}
\label{cusp1/2}
Let $E_{k}(\tau)$, $E_{k,\chi_{-2}}^{\infty}(\tau)$, $E_{k,\chi_{-2}}^{0}(\tau)$, $E_{k,\chi_{-4}}^{\infty}(\tau)$ and $E_{k,\chi_{-4}}^{0}(\tau)$ be as defined in Theorem \ref{main1}. Then under the transformation $\tau\to\tau+\frac{1}{2}$, the following identities hold.
\begin{align*}
\intertext{When $k$ is even,}
E_{k}\left(\tau+\frac{1}{2}\right)&=-E_{k}(\tau)+(2^{k}+2)E_{k}(2\tau)-2^{k}E_{k}(4\tau);\\
\intertext{when $k$ is odd,}
E_{k,\chi_{-2}}^{\infty}\left(\tau+\frac{1}{2}\right)&=-E_{k,\chi_{-2}}^{\infty}(\tau)+2E_{k,\chi_{-2}}^{\infty}(2\tau),\\
E_{k,\chi_{-2}}^{0}\left(\tau+\frac{1}{2}\right)&=-E_{k,\chi_{-2}}^{0}(\tau)+2^{k}E_{k,\chi_{-2}}^{0}(2\tau),\\
E_{k,\chi_{-4}}^{\infty}\left(\tau+\frac{1}{2}\right)&=-E_{k,\chi_{-4}}^{\infty}(\tau)+2E_{k,\chi_{-4}}^{\infty}(2\tau),\\
E_{k,\chi_{-4}}^{0}\left(\tau+\frac{1}{2}\right)&=-E_{k,\chi_{-4}}^{0}(\tau)+2^{k}E_{k,\chi_{-4}}^{0}(2\tau).
\end{align*}
\end{Lemma}

\begin{proof}
The proofs are similar to that of \cite[Lemma 3.2]{cky}, so we omit the details.
\end{proof}

\begin{Lemma}
\label{ord}
Let ${\rm ord}_{z}(f)$ denote the order of vanishing of $f(\tau)$ at $z$. Let $F_{k,2}(\tau)$ and $F_{k,4}(\tau)$ be as defined in Theorem \ref{main1}. Then we have
\begin{align}
\label{8ord}
{\rm ord}_{1/2}\left(F_{k,2}\right)&=\begin{cases}\frac{1}{2}&\mbox{if $k$ is odd,}\\1&\mbox{if $k$ is even,}\end{cases}\\
\label{16ord}
{\rm ord}_{1/2}\left(F_{k,4}\right)&=\begin{cases}2&\mbox{if $k$ is odd,}\\1&\mbox{if $k$ is even.}\end{cases}
\end{align}
\end{Lemma}

\begin{proof}
By the well-known transformation formulas for Let $E_{k}(\tau)$, $E_{k,\chi_{-2}}^{\infty}(\tau)$ and $E_{k,\chi_{-2}}^{0}(\tau)$, see e.g., \cite{kolberg} and \cite[pp. 79--83]{serre}, and Lemma \ref{cusp1/2}, we can deduce that 
for $k$ odd,
\begin{align*}
&\left.\left[E_{k,\chi_{-2}}^{\infty}(\tau)+(-8)^{(k-1)/2}E_{k,\chi_{-2}}^{0}(\tau)\right]\right|_{k}\begin{pmatrix}1&0\\2&1\end{pmatrix}\\
&=(2\tau+1)^{-k}\left[E_{k,\chi_{-2}}^{\infty}\left(\frac{1}{2}+\frac{-1}{4\tau+2}\right)+(-8)^{(k-1)/2}E_{k,\chi_{-2}}^{0}\left(\frac{1}{2}+\frac{-1}{4\tau+2}\right)\right]\\
&=(2\tau+1)^{-k}\Bigg\{-E_{k,\chi_{-2}}^{\infty}\left(\frac{-1}{4\tau+2}\right)+2E_{k,\chi_{-2}}^{\infty}\left(\frac{-1}{2\tau+1}\right)\\
&\quad\qquad\qquad\qquad+(-8)^{(k-1)/2}\left[-E_{k,\chi_{-2}}^{0}\left(\frac{-1}{4\tau+2}\right)+2^{k}E_{k,\chi_{-2}}^{0}\left(\frac{-1}{2\tau+1}\right)\right]\Bigg\}\\
&=\frac{i2^{k}}{8^{1/2}}E_{k,\chi_{-2}}^{0}\left(\frac{\tau}{2}+\frac{1}{4}\right)-\frac{2i}{8^{1/2}}E_{k,\chi_{-2}}^{0}\left(\frac{\tau}{4}+\frac{1}{8}\right)\\
&\qquad+\frac{i8^{1/2}(-8)^{(k-1)/2}}{4^{k}}\left[E_{k,\chi_{-2}}^{\infty}\left(\frac{\tau}{2}+\frac{1}{4}\right)-E_{k,\chi_{-2}}^{\infty}\left(\frac{\tau}{4}+\frac{1}{8}\right)\right]\\
&=Cq^{1/4}+O(q^{1/2})
\end{align*}
for some nonzero constant $C$ as $\tau\to i\infty$.
For $k=2$,
\begin{align*}
&\left[-E_{2}(\tau)+E_{2}(2\tau)-2E_{2}(4\tau)+8E_{2}(8\tau)\right]\Bigg|_{2}\begin{pmatrix}1&0\\2&1\end{pmatrix}\\
&=(2\tau+1)^{-2}\Bigg[-E_{2}\left(\frac{1}{2}+\frac{-1}{4\tau+2}\right)+E_{2}\left(1+\frac{-1}{2\tau+1}\right)\\
&\qquad\qquad\qquad\qquad-2E_{2}\left(2+\frac{-1}{(2\tau+1)/2}\right)+8E_{2}\left(4+\frac{-1}{(2\tau+1)/4}\right)\Bigg]\\
&=(2\tau+1)^{-2}\Bigg[E_{2}\left(\frac{-1}{4\tau+2}\right)-5E_{2}\left(\frac{-1}{2\tau+1}\right)\\
&\qquad\qquad\qquad\qquad+2E_{2}\left(\frac{-1}{(2\tau+1)/2}\right)+8E_{2}\left(\frac{-1}{(2\tau+1)/4}\right)\Bigg]\\
&=4E_{2}\left({4\tau+2}\right)+\frac{12}{\pi i(2\tau+1)}-5E_{2}\left({2\tau+1}\right)-\frac{30}{\pi i(2\tau+1)}\\
&\qquad+\frac{1}{2}E_{2}\left(\tau+\frac{1}{2}\right)+\frac{6}{\pi i(2\tau+1)}+\frac{1}{2}E_{2}\left(\frac{\tau}{2}+\frac{1}{4}\right)+\frac{12}{\pi i(2\tau+1)}\\
&=4E_{2}\left({4\tau+2}\right)-5E_{2}\left({2\tau+1}\right)+\frac{1}{2}E_{2}\left(\tau+\frac{1}{2}\right)+\frac{1}{2}E_{2}\left(\frac{\tau}{2}+\frac{1}{4}\right)\\
&=Cq^{1/2}+O(q)
\end{align*}
for some nonzero constant $C$ as $\tau\to i\infty$, and for $k\geq4$ and even,
\begin{align*}
&\left.\left[(-1)^{\frac{k}{2}}E_{k}(\tau)-(-1)^{\frac{k}{2}}E_{k}(2\tau)-2^{\frac{k}{2}}E_{k}(4\tau)+8^{\frac{k}{2}}E_{k}(8\tau)\right]\right|_{k}\begin{pmatrix}1&0\\2&1\end{pmatrix}\\
&=(2\tau+1)^{-k}\Bigg[(-1)^{\frac{k}{2}}E_{k}\left(\frac{1}{2}+\frac{-1}{4\tau+2}\right)-(-1)^{\frac{k}{2}}E_{k}\left(1+\frac{-1}{2\tau+1}\right)\\
&\qquad\qquad\qquad\qquad-2^{\frac{k}{2}}E_{k}\left(2+\frac{-1}{(2\tau+1)/2}\right)+8^{\frac{k}{2}}E_{k}\left(4+\frac{-1}{(2\tau+1)/4}\right)\Bigg]\\
&=(2\tau+1)^{-k}\Bigg[-(-1)^{\frac{k}{2}}E_{k}\left(\frac{-1}{4\tau+2}\right)+(-1)^{\frac{k}{2}}(2^{k}+2)E_{k}\left(\frac{-1}{2\tau+1}\right)\\
&\qquad\qquad\qquad\qquad-(-1)^{\frac{k}{2}}2^{k}E_{k}\left(\frac{-1}{(2\tau+1)/2}\right)-(-1)^{\frac{k}{2}}E_{k}\left(\frac{-1}{2\tau+1}\right)\\
&\qquad\qquad\qquad\qquad-2^{\frac{k}{2}}E_{k}\left(\frac{-1}{(2\tau+1)/2}\right)+8^{\frac{k}{2}}E_{k}\left(\frac{-1}{(2\tau+1)/4}\right)\Bigg]\\
&=-(-1)^{\frac{k}{2}}2^{k}E_{k}(4\tau+2)+(-1)^{\frac{k}{2}}(2^{k}+2)E_{k}(2\tau+1)-(-1)^{\frac{k}{2}}E_{k}\left(\tau+\frac{1}{2}\right)\\
&\qquad-(-1)^{\frac{k}{2}}E_{k}(2\tau+1)-2^{-\frac{k}{2}}E_{k}\left(\tau+\frac{1}{2}\right)+2^{-\frac{k}{2}}E_{k}\left(\frac{\tau}{2}+\frac{1}{4}\right)\\
&=Cq^{1/2}+O(q)
\end{align*}
for some nonzero constant $C$ as $\tau\to i\infty$. Together with the fact that the width of $\frac{1}{2}$ of $\Gamma_{0}(8)$ is $2$, the above observations conclude \eqref{8ord}. 

Making use of corresponding transformation formulas of $E_{k,\chi_{-4}}^{\infty}(\tau)$ and $E_{k,\chi_{-4}}^{0}(\tau)$ and Lemma \ref{cusp1/2} together with the fact that the width of $\frac{1}{2}$ of $\Gamma_{0}(16)$ is 4, we can prove \eqref{16ord} in a similar fashion, so we omit the details. 
\end{proof}

\begin{Lemma}
\label{etap}
If $f(\tau)=\prod_{d|N}\eta_{d}^{r_{d}}$ for some positive integer $N$ with $k=\frac{1}{2}\sum_{d|N}r_{d}\in\mathbb{Z}$, with the additional properties that 
$$
\sum_{d|N}dr_{d}\equiv0\pmod{24}
$$
and
$$
\sum_{d|N}\frac{N}{d}r_{d}\equiv0\pmod{24},
$$
then $f(\tau)$ satisfies
$$
f\left(\frac{a\tau+b}{c\tau+d}\right)=\chi(d)(c\tau+d)^{k}f(\tau)
$$
for every $\begin{pmatrix}a&b\\c&d\end{pmatrix}\in\Gamma_{0}(N)$. Here the character $\chi$ is defined by Jacobi symbol $\chi(d)=\left(\frac{(-1)^{k}s}{d}\right)$ where $s=\prod_{d|N}d^{r_{d}}$.
\end{Lemma}
\begin{proof}
See Gordon and Hughes \cite{goh}, or Newman \cite{new1,new3}.
\end{proof}

\begin{Lemma}
\label{order}
Let $a$, $c$ and $N$ be positive integers with $c|N$ and $\gcd(a,c)=1$. If $f(\tau)=\prod_{d|N}\eta_{d}^{r_{d}}$ satisfies the conditions of Lemma \ref{etap} for $N$, then the order of vanishing $\mathrm{ord}_{a/c}(f)$ of $
f(\tau)$ at the cusp $a/c$  is 
$$
\frac{N}{24}\sum_{d|N}\frac{\gcd(c,d)^{2}r_{d}}{\gcd(c,N/c)cd}.
$$
\end{Lemma}
\begin{proof}
See Biagioli \cite{bi}, Ligozat \cite{lig} or Martin \cite{ma}.
\end{proof}

\begin{proof}[Proof of Theorem \ref{main1}]
Let $m\in\{2,4\}$. Let $\ell_{m}$ be as defined in Theorem \ref{main}.
Consider the functions
$$
f_{m}(\tau)=\frac{F_{k,m}(\tau)}{\left(\theta(\tau)\theta(m\tau)\right)^{k}x_{m}(\tau)^{\ell_{m}}}\quad\mbox{and}\quad
g_{m}(\tau)=\frac{1}{x_{m}(\tau)}.
$$
Both $f_{m}(\tau)$ and $g_{m}(\tau)$ are analytic on the upper half plane $\mathbb{H}$. {Employing the transformation formulas for $E_{k}(\tau)$, $E_{k,\chi_{-2}}^{\infty}(\tau)$, $E_{k,\chi_{-2}}^{0}(\tau)$, $E_{k,\chi_{-4}}^{\infty}(\tau)$ and $E_{k,\chi_{-4}}^{0}(\tau)$, see e.g., \cite{kolberg} and \cite[pp. 79--83]{serre}, and Lemma \ref{etap} we may verify that both $f_{m}(\tau)$ and $g_{m}(\tau)$ are invariant under $\Gamma_{0}(4m)$ and $\begin{pmatrix}0&1\\-4m&0\end{pmatrix}$.
Therefore, both $f_{m}(\tau)$ and $g_{m}(\tau)$ are invariant under $\Gamma_{0}(4m)^{+}$, the group obtained from $\Gamma_{0}(4m)$ by adjoining its Fricke involution $\begin{pmatrix}0&1\\-4m&0\end{pmatrix}$.}  Let us analyze the behavior at $\tau=\infty$. By observing the $q$-expansions, we find that $f_{m}(\tau)$ has rational coefficients, and
$$
f_{m}(\tau)=\frac{1+O(q)}{(1+O(q))^{k}q^{\ell_{m}}(1+O(q))^{\ell_{m}}}=q^{-\ell_{m}}+O(q^{-\ell_{m}+1}).
$$
Therefore $f_{m}(\tau)$ has a pole of order $\ell_{m}$ at $\infty$. Similarly, we note that $g_{m}(\tau)$ has a simple pole at $\tau=\infty$.  It implies that there exist rational constants $a_{1,k,m},\,\ldots,a_{\ell_{m},k,m}$ such that the function
$$
h_{m}(\tau):=f_{m}(\tau)-\sum_{j=1}^{\ell_{m}}a_{j,k,m}g_{m}(\tau)^{j}
$$
has no pole at $\tau=\infty$, that is,
$$
h_{m}(\tau)=a_{0,k,m}+O(q)\quad\mbox{as $\tau\to \infty$}
$$
for some constant $a_{0}$. Let us consider the behavior of $h_{m}(\tau)$ at $\tau=\frac{1}{2}$. By Lemma \ref{ord}, we have
$$
\mbox{ord}_{1/2}(F_{k,m})=\begin{cases}\frac{1}{2}&\mbox{if $m=2$ and $k$ is odd,}\\1&\mbox{if $m=2$ and $k$ is even,}\\2&\mbox{if $m=4$ and $k$ is odd,}\\ 1&\mbox{if $m=4$ and $k$ is even.}\end{cases}
$$
Moreover, by Lemma \ref{order} together with the $\eta$-quotient representation \eqref{thetaeta} of $\theta(\tau)$,
$$
\theta(\tau)=\frac{\eta_{2}^{5}}{\eta_{1}^{2}\eta_{4}^{2}},
$$
 we can compute that 
$$
\mbox{ord}_{\mathfrak{c}_{m}}\left(\theta(\tau)\theta(m\tau)\right)=\begin{cases}\frac{1}{2}&\mbox{if $m=2$ and $\mathfrak{c}_{m}=\frac{1}{2}$,}\\1&\mbox{if $m=4$ and $\mathfrak{c}_{m}=\frac{1}{2}$,}\\
0&\mbox{if $m=4$ and $\mathfrak{c}_{m}=\frac{1}{4}$,}
\end{cases}
$$
and 
$$
\mbox{ord}_{\mathfrak{c}_{m}}(x_{m})=\begin{cases}-1&\mbox{if $m=2$ or $4$ and $\mathfrak{c}_{m}=\frac{1}{2}$,}\\
0&\mbox{if $m=4$ and $\mathfrak{c}_{m}=\frac{1}{4}$.}
\end{cases}
$$
Here $\mathfrak{c}_{m}$ denotes a cusp of $\Gamma_{0}(4m)$.
It is clear that ${\rm ord}_{\mathfrak{c}_{m}}(F_{k,m})\geq0$ for $m=4$ and $\mathfrak{c}_{m}=\frac{1}{4}$. Thus we summarize that
$$
\mbox{ord}_{\mathfrak{c}_{m}}(h_{m})=0\quad\mbox{if $m=2$ or $4$ and $\mathfrak{c}_{m}=\frac{1}{2}$,}
$$
and 
$$
\mbox{ord}_{\mathfrak{c}_{m}}(h_{m})\geq0\quad\mbox{if $m=4$ and $\mathfrak{c}_{m}=\frac{1}{4}$.}
$$
{{Since the set of inequvalent cusps of $\Gamma_{0}(4m)^+$ is  $\{\infty,\,\frac{1}{2}\}$ if $m=2$, or is $\{\infty,\,\frac{1}{2},\,\frac{1}{4}\}$ if $m=4$, it follows that $h_{m}(\tau)$}} is
 holomorphic on $X(\Gamma_{0}(4m)^+)=\Gamma_{0}(4m)^{+}\backslash\mathbb{H}$, and thus $h_{m}(\tau)$ is a constant, that is, $h_{m}(\tau)\equiv a_{0,k,m}$. Moreover, since $\mbox{ord}_{1/2}(h_{m})=0$, $h_{m}(\tau)$ does not vanish at $\frac{1}{2}$ and thus $a_{0,k,m}\ne0$. Therefore, we have
$$
f_{m}(\tau)=\sum_{j=0}^{\ell_{m}}a_{j,k,m}g_{m}(\tau)^{j},
$$
which is equivalent to
$$
F_{k,m}(\tau)=\left(\theta(\tau)\theta(m\tau)\right)^{k}\sum_{j=0}^{\ell_{m}}a_{j,k,m}x_{m}^{\ell-j}=\left(\theta(\tau)\theta(m\tau)\right)^{k}\sum_{j=0}^{\ell_{m}}b_{j,k,m}x_{m}^{j},
$$
where $b_{j,k,m}=a_{\ell_{m}-j,k,m}$. Equating the constant term shows that $b_{0,k,m}=1$. 
{{Now take $b_{j,k,m}=-c_{j,k,m}$ to complete the proof.}}
\end{proof}

\section{Concluding Remarks}
\label{conrem}
In this section, we conclude this work with some remarks on the essence of existence of  \eqref{eq1} and \eqref{eq2}, and an explanation of why the upper indices of the sums on the right hand sides of \eqref{eq1} and \eqref{eq2} cannot be improved further according to the functions $x_{m}(\tau)$ we use.

\begin{enumerate}
\item{
 The essence of existence of \eqref{eq1} and \eqref{eq2} is that for $l\in\{2,4\}$, $X(\Gamma_{0}(4m)^{+})$ is of genus zero, and the function $\frac{1}{x_{m}(\tau)}$ is a Hauptmodul of $X(\Gamma_{0}(4m)^{+})$, i.e., a generator of the function field $\mathbb{C}(X(\Gamma_{0}(4m)^{+}))$. Since the function $\frac{F_{k,m}(\tau)}{(\theta(\tau)\theta(m\tau))^{k}}$ is in $\mathbb{C}(X(\Gamma_{0}(4m)^{+}))$, and the locations of poles of it are the same as the locations of zeros of $\frac{1}{x_{m}(\tau)}$, then  $\frac{F_{k,m}(\tau)}{(\theta(\tau)\theta(m\tau))^{k}}$ is a polynomial in $x_{m}(\tau)$.

In general, for a positive integer $m$, we may obtain identities for $(\theta(\tau)\theta(m\tau))^{k}$ similar to \eqref{eq1} and \eqref{eq2} if we could construct a function $F_{k}(\tau)$ by a linear combination of Eisenstein series of weight $k$ such that $\frac{F_{k}(\tau)}{(\theta(\tau)\theta(m\tau))^{k}}$ is invariant under some genus zero discrete subgroup $\Gamma$ of $\mbox{SL}_{2}(\mathbb{R})$, and could construct a generator $\psi(\tau)$ for the function field $\mathbb{C}(X(\Gamma))$ such that the locations of zeros of $\psi(\tau)$ are the same as that of $(\theta(\tau)\theta(m\tau))^{k}$.\\
}

\item{We now explain that with the Hauptmodul $\frac{1}{x_{m}(\tau)}$ we use in this work, the upper indices $\ell_{m}$ cannot be improved any further, i.e., cannot be smaller. From the proof of Theorem \ref{main1}, we can first note that for $k$ fixed, the size of $\ell_{m}$ is determined by the order of vanishing of the function $F_{k,m}(\tau)$ at $\frac{1}{2}$; the higher the order is, the smaller $\ell_{m}$ will be. Then according to the proof of Lemma \ref{ord}, $\mbox{ord}_{1/2}(F_{k,m})$ is directly related to the definition of $F_{k,m}(\tau)$ as a linear combination of Eisenstein series. Thus it is natural ask whether one could redefine $F_{k,m}(\tau)$ to have higher order at $\frac{1}{2}$. In our cases, this is impossible. For example, for $m=2$, and $k\geq4$ and even, first we know that $(\theta(\tau)\theta(2\tau))^{k}$ and $x_{2}(\tau)$ are modular forms of weight $k$ on $\Gamma_{0}(8)$ with trivial character by Lemma \ref{etap}. Then we must have $F_{k,2}(\tau)$ be a linear combination of Eisenstein series of weight $k$ on $\Gamma_{0}(8)$ with trivial character, thus we must have
$$
F_{k,2}(\tau)=C_{1}E_{k}(\tau)+C_{2}E_{k}(2\tau)+C_{3}E_{k}(4\tau)+C_{4}E_{k}(8\tau)
$$
for some constants $C_{1},\,\ldots,C_{4}$ since $E_{k}(m\tau)$ for $m\in\{1,2,4,8\}$ are linearly independent Eisenstein series of weight $k$ on $\Gamma_{0}(8)$ with trivial character and the dimension of the space spanned by such Eisenstein series is $4$.
According to the proof of Lemma \ref{ord}, in order to have $\mbox{ord}_{1/2}(F_{k,2})\geq1$, we must have $C_{4}=0$. In addition, since $x_{2}(\tau)$ is invariant under $\Gamma_{0}(8)^{+}$, then $\frac{F_{k,2}(\tau)}{(\theta(\tau)\theta(2\tau))^{k}}$ must also be invariant under $\Gamma_{0}(8)^{+}$. Following such modularity, we can deduce that $C_{1}=0$ and $C_{2}=2^{-k/2}C_{3}$, and thus we have
$$
F_{k,2}(\tau)=C_{3}\left(2^{-k/2}E_{k}(2\tau)+E_{k}(4\tau)\right).
$$
However, similar to the proof of Lemma \ref{ord}, we can show that the order of vanishing of $2^{-k/2}E_{k}(2\tau)+E_{k}(4\tau)$ at $\frac{1}{2}$ is 0. This demonstrates our claim for the case $m=2$, and $k\geq4$ and even. For the cases $m=2$ and $k=2$, and $m=2$ or $4$ and $k=1$, the formulas we obtained do not involve any lower order term.  For the other cases, similar arguments can be applied by the facts that 
\begin{enumerate}[(i)]

\item{the space of Eisenstein series of weight $2$ on $\Gamma_0(16)$ is spanned by $mE_{2}(m\tau)-E_{2}(\tau)$ for $m|16$ and $m\ne1$, and $\sum_{n=1}^{\infty}\left(\frac{-4}{n}\right)\sigma_{1}(n)q^{n}$;
}

\item{the space of Eisenstein series of even weight $k\geq4$ on $\Gamma_0(16)$ is spanned by $E_{k}(m\tau)$ for $m|16$ and $\sum_{n=1}^{\infty}\chi_{-4}(n)\sigma_{k-1}(n)q^{n}$;
}

\item{the space of Eisenstein series of odd weight $k\geq3$ on $\Gamma_0(8)$ with character $\chi_{-2}$ is spanned by $E_{k,\chi_{-2}}^{\infty}(\tau)$ and $E_{k,\chi_{-2}}^{0}(\tau)$;}

\item{the space of Eisenstein series of odd weight $k\geq3$ on $\Gamma_0(16)$ with character $\chi_{-4}$ is spanned by $E_{k,\chi_{-4}}^{\infty}(m\tau)$ and $E_{k,\chi_{-4}}^{0}(m\tau)$ for $m|4$.}
\end{enumerate}
}

\end{enumerate}


\begin{thebibliography}{99}



\bibitem{ala}{A. Alaca, S. Alaca, M. F. Lemire and K. S. Williams, \emph{Nineteen quaternary quadratic forms,} Acta Arithmetica {\bf130} (2007), 277--310.}

\bibitem{ala2}{A. Alaca, S. Alaca, F. Uygul and K. S. Williams, \emph{Representations by sextenary quadratic forms whose coefficients are 1, 2 and 4,} Acta Arithmetica {\bf141} (2010), 289--309.}

\bibitem{ala3}{A. Alaca, S. Alaca and K. S. Williams, \emph{Fourteen octonary quadratic forms,} International Journal of Number Theory {\bf6} (2010), 37--50.}

\bibitem{ala4}{A. Alaca, S. Alaca and K. S. Williams, \emph{Sextenary quadratic forms and an identity of Klein and Fricke}, International Journal of Number Theory 6 (2010), 169--183.}

\bibitem{ala5}{A. Alaca, S. Alaca and K. S. Williams, \emph{The convolution sum $\sum_{l+16m=n}\sigma(l)\sigma(m)$,} Canadian Mathematical Bulletin 51 (2008), 3--14.}

\bibitem{ba}{P. Bachmann, \emph{Niedere Zahlentheorie,} Chelsea, New York, 1968.}


\bibitem{bru}{B. C. Berndt, \emph{Ramanujan's notebooks. Part III}, Springer-Verlag, New York, 1991.}

\bibitem{bru2}{B. C. Berndt, \emph{Ramanujan's notebooks. Part IV}, Springer-Verlag, New York, 1994.}

\bibitem{bi}{A. J. F. Biagioli, \emph{The construction of modular forms as products of transforms of the Dedekind eta function,} Acta Arith. {\bf54} (1990), 272--300.}



\bibitem{cooper}{S. Cooper, \emph{On sums of an even number of squares, and an even number of triangular numbers: an elementary approach based on Ramanujan's ${}_{1}\psi_{1}$ summation formula,} $q$-series with applications to combinatorics, number theory, and physics (Urbana, IL, 2000), 115–137, Contemp. Math., 291, Amer. Math. Soc., Providence, RI, 2001.}

\bibitem{cky}{S. Cooper, B. Kane and D. Ye, \emph{Analogues of the Ramanujan-Mordell Theorem,} J. Math. Anal. Appl. {\bf446} (2017), 568-579. }



\bibitem{eis}{G. F. Eisenstein, \emph{Mathematische Werke,} 2nd ed., Chelsea, New York, 1988.}

\bibitem{gla}{J. W. L. Glaisher, \emph{On the numbers of representations of a number as a sum of $2r$ squares, where $2r$ does not exceed eighteen,} Proc. London Math. Soc. {\bf5} (1907), 479--490.}

\bibitem{goh}{B. Gordon and K. Hughes, \emph{Multiplicative properties of $\eta-$products II,} A tribute to Emil Grosswald: Number Theory and related analysis, Cont. Math. of the Amer. Math. Soc. {\bf143} (1993), 415--430.}

\bibitem{hard}{G. H. Hardy and E. M. Wright, \emph{An Introduction to the Theory of Numbers,} Fourth
Edition, Clarendon Press, Oxford, 1960.}

\bibitem{jaco}{C. G. J. Jacobi, \emph{Gesammelte Werke. B$\ddot{a}$nde I,} Herausgegeben auf Veranlassung der K$\ddot{\mbox{o}}$niglich Preussischen Akademie der Wissenschaften. Zweite Ausgabe, Chelsea, New York, 1969.}

\bibitem{kolberg}{O. Kolberg, \emph{Note on the Eisenstein series of $\Gamma_{0}(p)$}, Arbok Univ. Bergen Mat.-Natur. Ser. {\bf1968} (1969), 20 pp.}

\bibitem{lig}{G. Ligozat, \emph{Courbes modulaires de genre 1,} Bull. Soc. Math. France [Memoire 43] (1972), 1--80.}

\bibitem{liou}{J. Liouville, \emph{Nombre des repr$\acute{e}$sentations d'un entier quelconque sous la forme d'une comme de dix carr$\acute{e}$s,} J. Math. Pures Appl. {\bf11} (1966), 1--8.}

\bibitem{liou3}{J. Liouville, \emph{Sur la forme $x^{2}+y^{2}+2(z^{2}+t^{2})$,} ibid. 5 (1860), 269--272.}

\bibitem{liou2}{J. Liouville, \emph{Sur la forme $x^{2}+y^{2}+4(z^{2}+t^{2})$,} ibid. 5 (1860), 305--308.}



\bibitem{ma}{Y. Martin, \emph{Multiplicative $\eta$-quotients,} Trans. Amer. Math. Soc. {\bf348} (1996), 4825--4856.}

\bibitem{mordell}{L. J. Mordell, \emph{On the representation of numbers as the sum of $2$r squares,} Quart. J. Pure and Appl. Math., Oxford {\bf48} (1917), 93--104.}

\bibitem{new1}{M. Newman, \emph{Construction and application of a certain class of modular functions,} Proc. London Math. Soc. {\bf7} (1956), 334--350.}

\bibitem{new3}{M. Newman, \emph{Construction and application of a certain class of modular functions II,} Proc. London Math. Soc. {\bf9} (1959), 373--387.}

\bibitem{pe1}{T. Pepin, \emph{$\acute{E}$tude sur quelques formules d'analyse utiles dans la th$\acute{e}$orie des nombres,} Atti Accad. Pont. Nuovi Lincei {\bf38} (1884-5), 139--196.}

\bibitem{pe2}{T. Pepin, \emph{Sur quelques formes quadratiques quaternaires,} J. Math. Pures Appl. 6 (1890),
5--67.}


\bibitem{ram}{S. Ramanujan, \emph{On certain arithmetical functions,} Trans. Cambridge Philos. Soc. {\bf22} (1916), 159--184.}

\bibitem{ram2}{S. Ramanujan, \emph{Collected papers of Srinivasa Ramanujan,} edited by G. H. Hardy et al., AMS Chelsea, Providence, RI, 2000.}

\bibitem{serre}{J.-P. Serre, \emph{A Course in Arithmetic}, New York: Springer-Verlag, 1973. }

\bibitem{shen}{ L.-C. Shen, \emph{On a class of q-series related to quadratic forms,} Bull. Inst. Math, Acad. Sinica, 26 (1998), 111--126.}

\end{thebibliography}
\end{document}